\documentclass[oneside,reqno,10pt]{amsart}
\usepackage{amsfonts,amssymb,amsmath,amsthm,color}

\usepackage{hyperref}

\usepackage{float}
\floatstyle{plaintop}
\restylefloat{table}
\usepackage[tableposition=top]{caption}

\theoremstyle{plain} 
\newtheorem{theorem}             {Theorem} 

\newtheorem{lemma}  [theorem]{Lemma}

\theoremstyle{definition}
\newtheorem{definition}{Definition}

\theoremstyle{remark}

\DeclareMathOperator{\ord}{ord}
\DeclareMathOperator{\sgn}{sgn}
\def\modd#1 #2{#1\ \mbox{\rm (mod}\ #2\mbox{\rm )}}

\begin{document}

\author{Daniel Tsai}
\address{Nagoya University, Graduate School of Mathematics, 464-8602, Furocho, Chikusa-ku, Nagoya, Japan}
\email{shokuns@math.nagoya-u.ac.jp}

\title{The invariance of the type of a $v$-palindrome}

\begin{abstract}
The notion of a $v$-palindrome is recently introduced by the author. Later, the author defined the notion of the type of a $v$-palindrome $n$ with respect to a number $m$ which can be repeatedly concatenated to form $n$. We prove that this notion of type is invariant over all $m$.
\end{abstract}
\maketitle

\section{Introduction}
The notion of a $v$-\emph{palindrome} is recently introduced by the author \cite{tsai18,tsai}. We define it here again.

\begin{definition}
Suppose that the canonical factorization of an integer $n\geq1$ is
\begin{equation}
  n=p^{a_1}_1\cdots p^{a_s}_s q_1\cdots q_t,
\end{equation}
where $s,t\geq0$ and $a_1,\ldots,a_s\ge2$ are integers and $p_1,\ldots,p_s,q_1,\ldots,q_t$ are distinct primes. We define
\begin{equation}
  v(n) = \sum^s_{i=1}(p_i+a_i) + \sum^t_{j=1} q_j.
\end{equation}
\end{definition}
\begin{definition}
Suppose that the decimal representation of an integer $n\geq1$ is
\begin{equation}
  n=d_0+d_110+\cdots+d_{L-1}10^{L-1},
\end{equation}
where $L\geq1$ and $0\leq d_0,d_1,\ldots,d_{L-1}<10$ are integers with $d_{L-1}\neq0$. We define
\begin{equation}
  r(n) = d_{L-1}+d_{L-2}10+\cdots+d_{0}10^{L-1}.
\end{equation}
\end{definition}
\begin{definition}
  An integer $n\geq1$ is a $v$-\emph{palindrome} if $10\nmid n$, $n\neq r(n)$, and $v(n) = v(r(n))$.
\end{definition}
The sequence of $v$-palindromes is A338039 in OEIS \cite{oeis}. The main theorem in \cite{tsai} describes a periodic phenomenon pertaining to $v$-palindromes and repeated concatenations. We first give notation for repeated concatenations.
\begin{definition}
For integers $k,L\geq1$, we define
\begin{equation}
    \rho_{k,L}=\overbrace{
1\underbrace{0\ldots0}_\text{$L-1$}1
\underbrace{0\ldots0}_\text{$L-1$}1
\ldots
1\underbrace{0\ldots0}_\text{$L-1$}1
}^\text{$k$}.
  \end{equation}
  That is, there are $k$ ones and between any consecutive pair of them, $L-1$ zeros.
\end{definition}
\begin{definition}
  For integers $n,k\geq1$ be integers with $n$ having $L$ decimal digits, the $k$-copy \emph{repeated concatenation} of $n$ is $n(k) =n \rho_{k,L}$.
\end{definition}
Visually speaking, $n(k)$ is the number formed by writing the decimal digits of $n$ a total of $k$ times, one after another.
\begin{theorem}[{\cite[Theorem 1]{tsai}}]\label{thm:pphen}
Let $n\geq1$ be an integer with $10\nmid n$ and $n\neq r(n)$. There exists an integer $\omega\geq1$ such that for integers $k\geq1$, $n(k)$ is a $v$-palindrome if and only if $n(k+\omega)$ is.
\end{theorem}
Let $n$ be as in the above theorem, then we have the sequence of repeated concatenations
\begin{equation}\label{eq:rec}
  n, n(2), n(3),\ldots,n(k),\ldots.
\end{equation}
Some or none of the above terms are $v$-palindromes while the others are not. What Theorem \ref{thm:pphen} is saying is that the ``whether $v$-palindromic'' pattern in \eqref{eq:rec} is periodic. This pattern is made clearer by the author in \cite{tsai21}, in which the concept of the \emph{type} of a $v$-palindrome is defined. More precisely, when $n(k)$ is a $v$-palindrome, it will have a type with respect to $n$, denoted in this paper by $\mathbf{Type}(n(k),n)$. However, the same $v$-palindrome $n(k)$ might also be $m(j)$, where $m,j\geq1$ are integers. Thus $\mathbf{Type}(m(j),m)$ also exists. It was posed as a conjecture in \cite{tsai} that these types must agree, whose proof is the aim of this paper.

There is a general procedure devised in \cite{tsai21} which is to be applied to $n$ to clarify the ``whether $v$-palindromic'' pattern in \eqref{eq:rec}. We recall this procedure in Section \ref{sec:genproc} to the extent suitable for our proof of the invariance of type. In Section \ref{sec:genprocnk}, we apply the procedure of Section \ref{sec:genproc} to all the repeated concatenations $n(k)$ at once. Finally in Section \ref{sec:proof}, we prove the invariance of the type of a $v$-palindrome. In Section \ref{sec:pre}, we first gather the definitions, notation, and lemmas needed for the rest of the paper. Much of Sections \ref{sec:pre} and \ref{sec:genproc} are a very quick restatement of content in \cite{tsai21} and so might appear very unmotivated.

$\sgn$ denotes the \emph{sign function} defined by $\sgn(x)=1$ if $x>0$ and $\sgn(x)=-1$ if $x<0$. For a prime $p$ and integer $n\neq0$, $\ord_p(n)$ denotes the greatest integer $a$ such that $p^a\mid n$.

\section{Preliminaries}\label{sec:pre}

\subsection{The numbers $h_{p^\alpha,L}$}
We define certain numbers we denote by $h_{p^\alpha,L}$ and state two lemmas about them, proving the second one.
\begin{definition}
  For a prime power $p^\alpha$ with $p\notin\{2,5\}$ and integer $L\geq1$, $h_{p^\alpha,L}$ denotes the order of $10^L$ regarded as an element of the group $(\mathbb{Z}/p^{\alpha+\ord_p(10^L-1)}\mathbb{Z})^\times$.
\end{definition}
\begin{lemma}[{\cite[Lemma 1]{tsai}}]\label{lem:rhodiv}
  Let $p^\alpha$ be a prime power with $p\notin\{2,5\}$ and $L\geq1$ an integer. Then $h_{p^\alpha,L}\geq2$ and for integers $k\geq1$, $p^\alpha\mid \rho_{k,L}$ if and only if $h_{p^\alpha,L}\mid k$.
\end{lemma}
 \begin{lemma}\label{lem:lemh}
   Let $p^\alpha$ be a prime power with $p\notin\{2,5\}$ and $k,L\geq1$ integers. We have
   \begin{equation}\label{eq:formula}
     h_{p^\alpha,Lk} = \frac{h_{p^{\alpha+\ord_p(\rho_{k,L})},L}}{(k,h_{p^{\alpha+\ord_p(\rho_{k,L})},L})}.
   \end{equation}
 \end{lemma}
 \begin{proof}
    The number $h_{p^\alpha,Lk}$ is the smallest positive integer such that
    \begin{equation}\label{eq:cong}
      (10^{Lk})^{h_{p^\alpha,Lk}} \equiv 1\pmod{p^{\alpha + \ord_p(10^{Lk}-1)}}.
    \end{equation}
    We have
    \begin{equation}
      10^{Lk} -1 = (10^L-1)(10^{L(k-1)}+10^{L(k-2)}+\cdots+1) = (10^L-1)\rho_{k,L},
    \end{equation}
    thus
    \begin{equation}
      \ord_p(10^{Lk} -1) = \ord_p(10^L-1) + \ord_p(\rho_{k,L}).
    \end{equation}
    Hence we can rewrite \eqref{eq:cong} as
    \begin{equation}
      (10^{Lk})^{h_{p^\alpha,Lk}} \equiv 1\pmod{p^{(\alpha + \ord_p(\rho_{k,L}))+\ord_p(10^L-1)}}.
    \end{equation}
    Now the number $h_{p^{\alpha+\ord_p(\rho_{k,L})},L}$ is the smallest positive integer such that
    \begin{equation}
      (10^L)^{h_{p^{\alpha+\ord_p(\rho_{k,L})},L}} \equiv 1\pmod{p^{(\alpha + \ord_p(\rho_{k,L}))+\ord_p(10^L-1)}}.
    \end{equation}
    Hence by the property of cyclic groups, $h_{p^\alpha,Lk}$ is the smallest positive integer for which $h_{p^{\alpha+\ord_p(\rho_{k,L})},L}\mid kh_{p^\alpha,Lk}$, and this is clearly that given by \eqref{eq:formula}.
 \end{proof}
 
 \subsection{The functions $\varphi_{p,\delta}$} We define certain functions we denote by $\varphi_{p,\delta}\colon\mathbb{N}\cup\{0\}\to\mathbb{N}$ and further related concepts.
\begin{definition}\label{def:phipd}
We define functions $\varphi_{p,\delta}\colon\mathbb{N}\cup\{0\}\to\mathbb{N}$, for $p$ a prime and $\delta\geq1$ an integer, in three cases as follows.
  \begin{description}
    \item[{($(p,\delta) = (2,1)$)}] \begin{equation}
    \varphi_{2,1}(\alpha) = \begin{cases}
      2\quad\text{if $\alpha\in\{0,1\}$},\\
      1\quad \text{if $\alpha\geq2$}.
    \end{cases}
  \end{equation}
   \item[{($p\neq 2$ and $\delta=1$)}]
   \begin{equation}
    \varphi_{p,1}(\alpha) = \begin{cases}
      p\quad\text{if $\alpha=0$},\\
      2\quad\text{if $\alpha=1$},\\
      1\quad \text{if $\alpha\geq2$}.
    \end{cases}
  \end{equation}
  \item[{($\delta\geq2$)}]
   \begin{equation}
    \varphi_{p,\delta}(\alpha) = \begin{cases}
      p+\delta\quad\text{if $\alpha=0$},\\
      1+\delta\quad\text{if $\alpha=1$},\\
      \delta\quad \text{if $\alpha\geq2$}.
    \end{cases}
  \end{equation}
  \end{description}
\end{definition}
\begin{definition}
  For $p$ a prime and $\delta\geq1$ an integer, we put $R_{p,\delta} = \varphi_{p,\delta}(\mathbb{N}\cup\{0\})$.
\end{definition}
\begin{lemma}[{\cite[Lemma 3.1]{tsai21}}]\label{Cases}
  For an ordered quadruple $(p,\delta,u,\mu)$ of integers, where $p$ is a prime, $\delta\geq1$, $u\in R_{p,\delta}$, and $\mu\ge0$, exactly one of the following is the case.
  \renewcommand{\labelenumi}{[\roman{enumi}]}
  \begin{itemize}
    \item[{\rm [i]}] $ \varphi^{-1}_{p,\delta}(u)=\{0\}$ and $\mu=0$, or $\varphi^{-1}_{p,\delta}(u)=\{1\}$ and $\mu=1$, or $\varphi^{-1}_{p,\delta}(u)=\{0,1\}$ and $\mu=1$, \label{item1}
      \item[{\rm [ii]}] $\varphi^{-1}_{p,\delta}(u)=\{1\}$ and $\mu=0$,
      \item[{\rm [iii]}] $\varphi^{-1}_{p,\delta}(u)=\{0,1\}$ and $\mu=0$,
      \item[{\rm [iv]}] $\varphi^{-1}_{p,\delta}(u)=\mathbb{N}\setminus\{1\}$ and $\mu=1$,
      \item[{\rm [v]}] $\varphi^{-1}_{p,\delta}(u)=\mathbb{N}\setminus\{1\}$ and $\mu=0$,
      \item[{\rm [vi]}] $\varphi^{-1}_{p,\delta}(u)=\mathbb{N}\setminus\{1\}$ and $\mu\ge2$,
      \item[{\rm [vii]}] otherwise.
  \end{itemize}
\end{lemma}
\begin{definition}\label{def:D}
  For a quadruple $(p,\delta,u,\mu)$ as in Lemma \ref{Cases}, $D(p,\delta,u,\mu)$ denotes the case which holds. For example, $D(p,\delta,u,\mu)=[\mathrm{ii}]$ if and only if $\varphi^{-1}_{p,\delta}(u)=\{1\}$ and $\mu=0$, and $D(p,\delta,u,\mu)=[\mathrm{iv}]$ if and only if $\varphi^{-1}_{p,\delta}(u)=\mathbb{N}\setminus\{1\}$ and $\mu=1$, and etc.
\end{definition}

\subsection{The sets $S(A,B)$} We define certain sets we denote by $S(A,B)$.
\begin{definition}
  For finite sets $A,B\subseteq\mathbb{N}$, we put
  \begin{equation}
    S(A,B) = \{x\in\mathbb{Z}\colon (a\mid x\text{ for every }a\in A)\text{ and }(b\nmid x\text{ for every }b\in B)\}.
  \end{equation}
\end{definition}

\section{General procedure}\label{sec:genproc}
Throughout this section we fix an integer $n\geq1$ with $10\nmid n$, $n\neq r(n)$, and $L$ decimal digits and describe the general procedure associated to it.
\begin{description}
  \item[Step 1] Factorize both $n$ and $r(n)$,
  \begin{align}
    n&= p^{a_1}_1\cdots p^{a_m}_m,\\
    r(n)& = p^{b_1}_1\cdots p^{b_m}_m,
  \end{align}
  where $p<\cdots <p_m$ are primes, and $a_i,b_i\geq0$ are integers, not both $0$.
  \item[Step 2] Look for those primes $p_i$ for which $a_i\neq b_i$, they are called the \emph{crucial primes} (of $n$) and their set denoted by $K$ (or $K(n)$ if we wish to indicate $n$). Since $n\neq r(n)$, $K$ is a nonempty finite set. We are only going to focus on these primes, and so we denote them again by $p_1<\cdots<p_m$, and the exponents are $a_i,b_i$. Define the numbers $\delta_i = a_i-b_i$, $\mu_i = \min(a_i,b_i)$, for $1\leq i\leq m$.
  \item[Step 3] The \emph{characteristic equation} (for $n$) is
  \begin{equation}
    \sgn(\delta_1)u_1+\sgn(\delta_2)u_2+\cdots+\sgn(\delta_m)u_m = 0.
  \end{equation}
  We want to solve it for $u_i \in R_{p_i,|\delta_i|}$, and these solutions are called the \emph{characteristic solutions} (of $n$). The set of characteristic solutions is denoted by $\mathcal{U}$ (or $\mathcal{U}(n)$ if we wish to indicate $n$). If $\mathcal{U}=\varnothing$, then conclude that $c(n) = \infty$ and $\omega_0(n) = 1$. Otherwise, suppose that $\mathcal{U} = \{\mathbf{u}_1,\ldots,\mathbf{u}_t\}$, where we write $\mathbf{u}_l = (u_{li})^m_{i=1}$, for $1\leq l\leq t$.
  \item[Step 4] We make two tables of the crucial primes $p_i$ ($1\leq i\leq m$) versus the characteristic solutions $\mathbf{u}_l$ ($1\le l\le t$) as follows.
  
  The first is where in the $(p_i,\mathbf{u}_l)$-entry we have $D(p_i,|\delta_i|,u_{li},\mu_i)$, and is called the \emph{first table} (for $n$). We illustrate the generic first table as follows.
  \begin{table}[H]
 \caption{The first table.}
 \label{table:indicatorfunB}
 \centering
  \begin{tabular}{cccccc}
   \hline
   $$ & $\mathbf{u}_1$ & $\cdots$ & $\mathbf{u}_l$& $\cdots$& $\mathbf{u}_t$\\
   \hline \hline
   $p_1$ & $D(p_1,|\delta_1|,u_{11},\mu_1)$& $\cdots$& $D(p_1,|\delta_1|,u_{l1},\mu_1)$& $\cdots$& $D(p_1,|\delta_1|,u_{t1},\mu_1)$\\
   $\vdots$ & $\vdots$& $$& $\vdots$& $$ & $\vdots$\\
   $p_i$ & $D(p_i,|\delta_i|,u_{1i},\mu_i)$& $\cdots$& $D(p_i,|\delta_i|,u_{li},\mu_i)$& $\cdots$& $D(p_i,|\delta_i|,u_{ti},\mu_i)$\\
   $\vdots$ & $\vdots$& $$& $\vdots$& $$& $\vdots$\\
   $p_m$ & $D(p_m,|\delta_m|,u_{1m},\mu_m)$& $\cdots$& $D(p_m,|\delta_m|,u_{lm},\mu_m)$& $\cdots$& $D(p_m,|\delta_m|,u_{tm},\mu_m)$\\
   \hline
  \end{tabular}
\end{table}
The second is where in the $(p_i,\mathbf{u}_l)$-entry we have $T_{p_i,\mathbf{u}_l}=(A_{p_i,\mathbf{u}_l},B_{p_i,\mathbf{u}_l})$, which we now define. For a prime power $p^\alpha$ with $p\notin\{2,5\}$, we abbreviate $h_{p^\alpha,L}$ as $h_{p^\alpha}$. For $p_i\notin\{2,5\}$, we put
\begin{equation}\label{defnT1}
      T_{p_i,\mathbf{u}_l}=(A_{p_i,\mathbf{u}_l},B_{p_i,\mathbf{u}_l})=
      \begin{cases}
        (\varnothing,\{h_{p_i}\})& \text{if $D(p_i,|\delta_i|,u_{li},\mu_i)=[\mathrm{i}]$,} \\
        (\{h_{p_i}\},\{h_{p^2_i}\}) & \text{\text{if $D(p_i,|\delta_i|,u_{li},\mu_i)=[\mathrm{ii}]$,}} \\
        (\varnothing,\{h_{p^2_i}\}) & \text{\text{if $D(p_i,|\delta_i|,u_{li},\mu_i)=[\mathrm{iii}]$,}} \\
       (\{h_{p_i}\},\varnothing) & \text{\text{if $D(p_i,|\delta_i|,u_{li},\mu_i)=[\mathrm{iv}],$}} \\
        (\{h_{p^2_i}\},\varnothing)& \text{\text{if $D(p_i,|\delta_i|,u_{li},\mu_i)=[\mathrm{v}].$}}
      \end{cases}
    \end{equation}
    For $p_i\in\{2,5\}$, we put
    \begin{equation}\label{defnT2}
       T_{p_i,\mathbf{u}_l}=(A_{p_i,\mathbf{u}_l},B_{p_i,\mathbf{u}_l})=
      \begin{cases}
        (\varnothing,\varnothing)& \text{if $D(p_i,|\delta_i|,u_{li},\mu_i)=[\mathrm{i}]$,} \\
        (\varnothing,\{1\}) & \text{\text{if $D(p_i,|\delta_i|,u_{li},\mu_i)=[\mathrm{ii}]$,}} \\
        (\varnothing,\varnothing) & \text{\text{if $D(p_i,|\delta_i|,u_{li},\mu_i)=[\mathrm{iii}],$}} \\
       (\varnothing,\{1\}) & \text{\text{if $D(p_i,|\delta_i|,u_{li},\mu_i)=[\mathrm{iv}],$}} \\
        (\varnothing,\{1\})& \text{\text{if $D(p_i,|\delta_i|,u_{li},\mu_i)=[\mathrm{v}]$.}}
      \end{cases}
    \end{equation}
    For any $p_i$, we put
    \begin{equation}\label{defnT3}
    T_{p_i,\mathbf{u}_l}=(A_{p_i,\mathbf{u}_l},B_{p_i,\mathbf{u}_l})=
    \begin{cases}
      (\varnothing,\varnothing) & \text{if $D(p_i,|\delta_i|,u_{li},\mu_i)=[\mathrm{vi}]$,} \\
      (\varnothing,\{1\}) & \text{if $D(p_i,|\delta_i|,u_{li},\mu_i)=[\mathrm{vii}]$.}
    \end{cases}
    \end{equation}
This table of entries $T_{p_i,\mathbf{u}_l}$ is called the \emph{second table} (for $n$) and we illustrate the generic one as follows.
\begin{table}[H]
 \caption{The second table.}
 \label{table:indicatorfunB}
 \centering
  \begin{tabular}{cccccc}
   \hline
   $$ & $\mathbf{u}_1$ & $\cdots$ & $\mathbf{u}_l$& $\cdots$& $\mathbf{u}_t$\\
   \hline \hline
   $p_1$ & $(A_{p_1,\mathbf{u}_1},B_{p_1,\mathbf{u}_1})$& $\cdots$& $(A_{p_1,\mathbf{u}_l},B_{p_1,\mathbf{u}_l})$& $\cdots$& $(A_{p_1,\mathbf{u}_t},B_{p_1,\mathbf{u}_t})$\\
   $\vdots$ & $\vdots$& $$& $\vdots$& $$ & $\vdots$\\
   $p_i$ & $(A_{p_i,\mathbf{u}_1},B_{p_i,\mathbf{u}_1})$& $\cdots$& $(A_{p_i,\mathbf{u}_l},B_{p_i,\mathbf{u}_l})$& $\cdots$& $(A_{p_i,\mathbf{u}_t},B_{p_i,\mathbf{u}_t})$\\
   $\vdots$ & $\vdots$& $$& $\vdots$& $$& $\vdots$\\
   $p_m$ & $(A_{p_m,\mathbf{u}_1},B_{p_m,\mathbf{u}_1})$& $\cdots$& $(A_{p_m,\mathbf{u}_l},B_{p_m,\mathbf{u}_l})$& $\cdots$& $(A_{p_m,\mathbf{u}_t},B_{p_m,\mathbf{u}_t})$\\
   \hline
  \end{tabular}
\end{table}
The first table helps us to construct the second table because the determination of $T_{p_i,\mathbf{u}_l}$ depends on $D(p_i,|\delta_i|,u_{li},\mu_i)$.

For each $\mathbf{u}_l$, we define the sets
\begin{gather}
  A_{\mathbf{u}_l} =\bigcup^m_{i=1}A_{p_i,\mathbf{u}_{l}}, \quad B_{\mathbf{u}_l} =\bigcup^m_{i=1}B_{p_i,\mathbf{u}_{l}}, \\
  S_{\mathbf{u}_l} = S(A_{\mathbf{u}_l}, B_{\mathbf{u}_l}).
\end{gather}
Visually speaking, $A_{\mathbf{u}_l}$ and $B_{\mathbf{u}_l}$ are respectively the union of the left and right coordinates of the entries in the $\mathbf{u}_l$-column. We then define
\begin{equation}
  S = \bigsqcup^t_{l=1} S_{\mathbf{u}_l}
\end{equation}
(that the sets $S_{\mathbf{u}_l}$ are pairwise disjoint follows from \cite[Theorem 4.4]{tsai21}).
Then we have the following theorem.
\begin{theorem}[\cite{tsai21}]
  For $k\geq1$, the number $n(k)$ is a $v$-palindrome if and only if $k\in S$.
\end{theorem}
Since the $\mathbf{u}_l$ are pairwise disjoint, we can make the following definition.
\begin{definition}
Suppose that $n(k)$ is a $v$-palindrome, where $k\geq1$. The \emph{type} of $n(k)$ with respect to $n$ is the unique $\mathbf{u}\in \mathcal{U}(n)$ such that $k\in S_{\mathbf{u}}$ and we denote $\mathbf{Type}(n(k),n)=\mathbf{u}$.
\end{definition}

Actually some of the $S_{\mathbf{u}_l}$ might be empty, and so there will be no $v$-palindrome $n(k)$ of type $\mathbf{u}_l$ with respect to $n$. Those $\mathbf{u}_l$ for which $S_{\mathbf{u}_l}\neq \varnothing$ are called the \emph{nondegenerate} characteristic solutions (of $n$) and their set is denoted by $\mathcal{U}^\ast$ (or $\mathcal{U}^\ast(n)$ if we wish to indicate $n$).
  
\end{description}

\section{General procedure applied to all the $n(k)$}\label{sec:genprocnk}
Throughout this section we fix an integer $n\geq1$ with $10\nmid n$, $n\neq r(n)$, and $L$ decimal digits and apply the general procedure of Section \ref{sec:genproc} simultaneously to all the $n(k)$, for $k\geq1$.
\begin{description}
  \item[Step 1] Factorize both $n$ and $r(n)$,
  \begin{align}
    n&= p^{a_1}_1\cdots p^{a_m}_m,\\
    r(n)& = p^{b_1}_1\cdots p^{b_m}_m,
  \end{align}
  where $p<\cdots <p_m$ are primes, and $a_i,b_i\geq0$ are integers, not both $0$. For $k\geq1$, we abbreviate $\rho_{k,L}$ as $\rho_k$. Since $n(k)=n\rho_{k}$ and $r(n(k)) = r(n)(k) =r(n)\rho_{k}$,
  \begin{align}
    n(k)&= p^{a_1}_1\cdots p^{a_m}_m\rho_{k},\\
    r(n(k))& = p^{b_1}_1\cdots p^{b_m}_m\rho_{k}.
  \end{align}
  Hence we see that for $1\leq i\leq m$,
  \begin{align}
    \ord_{p_i}(n(k)) = a_i+\ord_{p_i}(\rho_{k}),\label{eq:step11}\\
    \ord_{p_i}(r(n(k))) = b_i + \ord_{p_i}(\rho_{k}),\label{eq:step12}
  \end{align}
  and that for any other prime $p$,
  \begin{equation}
    \ord_{p}(n(k)) = \ord_{p}(r(n(k))) = \ord_p(\rho_{k}).\label{eq:step13}
  \end{equation}
  
  \item[Step 2] In view of the equalities \eqref{eq:step11}, \eqref{eq:step12}, and \eqref{eq:step13} in Step 1, the crucial primes of $n(k)$ are the same as that of $n$. That is, $K(n(k)) = K(n)$, and we denote it simply by $K$. We are only going to focus on these primes, and so we denote them again by $p_1<\cdots<p_m$, and we redefine
  \begin{align}
    a_i=\ord_{p_i}(n),\quad b_i=\ord_{p_i}(r(n)),
  \end{align}
  for $1\leq i\leq m$.
  Define the numbers $\delta_i = a_i-b_i$, $\mu_i = \min(a_i,b_i)$, for $1\leq i\leq m$. Since we are applying the general procedure to all the $n(k)$, we also define, for $1\leq i\leq m$ and all $k\geq1$,
  \begin{align}
    a_{ik} &= \ord_{p_i}(n(k)) = a_i + \ord_{p_i}(\rho_{k}),\\
    b_{ik} &= \ord_{p_i}(r(n(k))) = b_i + \ord_{p_i}(\rho_{k}),\\
    \delta_{ik} &= a_{ik} - b_{ik} = a_i-b_i = \delta_i,\label{eq:deltainv}\\
    \mu_{ik} &=  \min(a_{ik},b_{ik}) = \min(a_i,b_i)+\ord_{p_i}(\rho_{k}) = \mu_i + \ord_{p_i}(\rho_{k}). 
  \end{align}
  Hence we see clearly how the $a_{ik},b_{ik},\delta_{ik},\mu_{ik}$ changes as $k$ increases. More precisely, we see that the $\delta_{ik}$ do not change and that the changes in the $a_{ik},b_{ik},\mu_{ik}$ depend only on $\ord_{p_i}(\rho_{k})$. We denote $x_{ik} = \ord_{p_i}(\rho_{k})$, for $1\leq i\leq m$ and all $k\geq1$.
  
  \item[Step 3] In view of \eqref{eq:deltainv}, the characteristic equation for all the $n(k)$ are the same and it is
  \begin{equation}
    \sgn(\delta_1)u_1+\sgn(\delta_2)u_2+\cdots+\sgn(\delta_m)u_m = 0.
  \end{equation}
  Moreover, the characteristic solutions for all the $n(k)$ are the same. That is,
  \begin{equation}
  \mathcal{U}(n) = \mathcal{U}(n(2))=\mathcal{U}(n(3))=\cdots=\mathcal{U}.
  \end{equation}
  If $\mathcal{U}=\varnothing$, conclude that for all $k\geq1$, $n(k)$ is not a $v$-palindrome. Otherwise, suppose that $\mathcal{U} = \{\mathbf{u}_1,\ldots,\mathbf{u}_t\}$, where we write $\mathbf{u}_l = (u_{li})^m_{i=1}$, for $1\leq l\leq t$.
  \item[Step 4] We illustrate the first and second tables for $n(k)$ as follows.
  \begin{table}[H]
 \caption{The first table for $n(k)$.}
 \label{table:indicatorfunB}
 \centering
  \begin{tabular}{cccccc}
   \hline
   $$ & $\mathbf{u}_1$ & $\cdots$ & $\mathbf{u}_l$& $\cdots$& $\mathbf{u}_t$\\
   \hline \hline
   $p_1$ & $D(p_1,|\delta_1|,u_{11},\mu_1+x_{1k})$& $\cdots$& $D(p_1,|\delta_1|,u_{l1},\mu_1+x_{1k})$& $\cdots$& $D(p_1,|\delta_1|,u_{t1},\mu_1+x_{1k})$\\
   $\vdots$ & $\vdots$& $$& $\vdots$& $$ & $\vdots$\\
   $p_i$ & $D(p_i,|\delta_i|,u_{1i},\mu_i+x_{ik})$& $\cdots$& $D(p_i,|\delta_i|,u_{li},\mu_i+x_{ik})$& $\cdots$& $D(p_i,|\delta_i|,u_{ti},\mu_i+x_{ik})$\\
   $\vdots$ & $\vdots$& $$& $\vdots$& $$& $\vdots$\\
   $p_m$ & $D(p_m,|\delta_m|,u_{1m},\mu_m+x_{mk})$& $\cdots$& $D(p_m,|\delta_m|,u_{lm},\mu_m+x_{mk})$& $\cdots$& $D(p_m,|\delta_m|,u_{tm},\mu_m+x_{mk})$\\
   \hline
  \end{tabular}
\end{table}
For the second table, a third subscript of $k$ is added to the $A_{p_i,\mathbf{u}_l}$ and $B_{p_i,\mathbf{u}_l}$ to indicate the dependence on $k$.
\begin{table}[H]
 \caption{The second table for $n(k)$.}
 \label{table:indicatorfunB}
 \centering
  \begin{tabular}{cccccc}
   \hline
   $$ & $\mathbf{u}_1$ & $\cdots$ & $\mathbf{u}_l$& $\cdots$& $\mathbf{u}_t$\\
   \hline \hline
   $p_1$ & $(A_{p_1,\mathbf{u}_1,k},B_{p_1,\mathbf{u}_1,k})$& $\cdots$& $(A_{p_1,\mathbf{u}_l,k},B_{p_1,\mathbf{u}_l,k})$& $\cdots$& $(A_{p_1,\mathbf{u}_t,k},B_{p_1,\mathbf{u}_t,k})$\\
   $\vdots$ & $\vdots$& $$& $\vdots$& $$ & $\vdots$\\
   $p_i$ & $(A_{p_i,\mathbf{u}_1,k},B_{p_i,\mathbf{u}_1,k})$& $\cdots$& $(A_{p_i,\mathbf{u}_l,k},B_{p_i,\mathbf{u}_l,k})$& $\cdots$& $(A_{p_i,\mathbf{u}_t,k},B_{p_i,\mathbf{u}_t,k})$\\
   $\vdots$ & $\vdots$& $$& $\vdots$& $$& $\vdots$\\
   $p_m$ & $(A_{p_m,\mathbf{u}_1,k},B_{p_m,\mathbf{u}_1,k})$& $\cdots$& $(A_{p_m,\mathbf{u}_l,k},B_{p_m,\mathbf{u}_l,k})$& $\cdots$& $(A_{p_m,\mathbf{u}_t,k},B_{p_m,\mathbf{u}_t,k})$\\
   \hline
  \end{tabular}
\end{table}
For each $\mathbf{u}_l$ and all $k\geq1$, we define the sets
\begin{gather}
  A_{\mathbf{u}_l,k} =\bigcup^m_{i=1}A_{p_i,\mathbf{u}_{l},k}, \quad B_{\mathbf{u}_l,k} =\bigcup^m_{i=1}B_{p_i,\mathbf{u}_{l},k}, \\
  S_{\mathbf{u}_l,k} = S(A_{\mathbf{u}_l,k}, B_{\mathbf{u}_l,k}).
\end{gather}
We then define
\begin{equation}
  S_k = \bigsqcup^t_{l=1} S_{\mathbf{u}_l,k}.
\end{equation}

The set of nondegenerate solutions of $n(k)$ is $\mathcal{U}^\ast(n(k))$. Although the crucial primes and characteristic solutions are the same for all the $n(k)$, the nondegenerate solutions are not in general.  
\end{description}

\section{Proof of the invariance}\label{sec:proof}
Again, fix an integer $n\geq1$ with $10\nmid n$, $n\neq r(n)$, and $L$ decimal digits throughout this section. Assume that we have applied the general procedure to $n$ and to the $n(k)$, for all $k\geq1$, as described in Sections \ref{sec:genproc} and \ref{sec:genprocnk}. Moreover, all those notation are inherited. It suffices to prove that if $n(kj)$ is a $v$-palindrome, where $k,j\geq1$, then the type of $n(kj)$ with respect to $n$ is the same as that with respect to $n(k)$. That is,
\begin{equation}
  \mathbf{Type}(n(kj),n) = \mathbf{Type}(n(kj),n(k)).
\end{equation}
Assume that $\mathbf{Type}(n(kj),n) = \mathbf{u}_l$. This is equivalent to saying that $kj\in S_{\mathbf{u}_l}$. We need to show that $\mathbf{Type}(n(kj),n(k)) = \mathbf{u}_l$ too, i.e., $j\in S_{\mathbf{u}_l,k}$. We illustrate the $\mathbf{u}_l$-column in the second tables for $n$ and $n(k)$ as follows.
\begin{table}[H]
 \caption{The $\mathbf{u}_l$-column in the second table for $n$.}
 \label{table:indicatorfunB}
 \centering
  \begin{tabular}{cc}
   \hline
   $$ & $\mathbf{u}_l$\\
   \hline \hline
   $p_1$ & $(A_{p_1,\mathbf{u}_l},B_{p_1,\mathbf{u}_l})$\\
   $\vdots$ & $\vdots$\\
   $p_i$ & $(A_{p_i,\mathbf{u}_l},B_{p_i,\mathbf{u}_l})$\\
   $\vdots$ &$\vdots$\\
   $p_m$ & $(A_{p_m,\mathbf{u}_l},B_{p_m,\mathbf{u}_l})$\\
   \hline
  \end{tabular}
\end{table}
Visually speaking, that $kj\in S_{\mathbf{u}_l}$ means exactly that $kj$ is divisible by every number which appears in the left coordinate of an entry in the above column, and indivisible by every number which appears in the right coordinate of an entry in the above column.
\begin{table}[H]
 \caption{The $\mathbf{u}_l$-column in the second table for $n(k)$.}
 \label{table:indicatorfunB}
 \centering
  \begin{tabular}{cc}
   \hline
   $$ & $\mathbf{u}_l$\\
   \hline \hline
   $p_1$ & $(A_{p_1,\mathbf{u}_l,k},B_{p_1,\mathbf{u}_l,k})$\\
   $\vdots$ & $\vdots$\\
   $p_i$ & $(A_{p_i,\mathbf{u}_l,k},B_{p_i,\mathbf{u}_l,k})$\\
   $\vdots$ &$\vdots$\\
   $p_m$ & $(A_{p_m,\mathbf{u}_l,k},B_{p_m,\mathbf{u}_l,k})$\\
   \hline
  \end{tabular}
\end{table}
Similarly, that $j\in S_{\mathbf{u}_l,k}$ means exactly that $j$ is divisible by every number which appears in the left coordinate of an entry in the above column, and indivisible by every number which appears in the right coordinate of an entry in the above column. Hence it suffices to prove that $j\in S(A_{p_i,\mathbf{u}_l,k},B_{p_i,\mathbf{u}_l,k})$, for all $1\leq i\leq m$.

We focus on an arbitrary $(A_{p_i,\mathbf{u}_l,k},B_{p_i,\mathbf{u}_l,k})$ and denote $p=p_i$, $\mathbf{u}=\mathbf{u}_l$, $\delta = \delta_i$, $u=u_{li}$, $\mu_i = \mu$, and $x_k=x_{ik}$, and $(A,B) = (A_{p_i,\mathbf{u}_l},A_{p_i,\mathbf{u}_l})$ and $(A_k,B_k) = (A_{p_i,\mathbf{u}_l,k},B_{p_i,\mathbf{u}_l,k})$. Hence we need to prove that $j\in S(A_{k},B_k)$. We divide our consideration into various cases, according to the prime $p$ and $D=D(p,|\delta|,u,\mu)$, in reference to how the second table for $n$ is determined via \eqref{defnT1}, \eqref{defnT2}, and \eqref{defnT3}. Just as the determination of $(A,B)$ depends on $D$, the determination of $(A_k,B_k)$ depends on $D_k = D(p,|\delta|,u,\mu+x_k)$. By Lemma \ref{lem:lemh}, if $p\notin\{2,5\}$, then
\begin{align}
     h_{p,Lk} &= \frac{h_{p^{1+\ord_p(\rho_{k,L})},L}}{(k,h_{p^{1+\ord_p(\rho_{k,L})},L})}=\frac{h_{p^{1+x_k},L}}{(k,h_{p^{1+x_k},L})},\\
     h_{p^2,Lk} &= \frac{h_{p^{2+\ord_p(\rho_{k,L})},L}}{(k,h_{p^{2+\ord_p(\rho_{k,L})},L})}=\frac{h_{p^{2+x_k},L}}{(k,h_{p^{2+x_k},L})}.
   \end{align}
   We will be using the above equalities in the following case analysis.
\begin{description}
  \item[{($p\notin\{2,5\}$ and $D=[\mathrm{i}]$)}] We have $(A,B) = (\varnothing,\{h_{p,L}\})$. Thus $h_{p,L}\nmid kj$, and so $h_{p,L}\nmid k$. Consequently, by Lemma \ref{lem:rhodiv}, $p\nmid \rho_{k,L}$, and so $x_k=0$. Therefore $D_k = [\mathrm{i}]$ too, and so $(A_k,B_k) = (\varnothing,\{h_{p,Lk}\})$. Since $x_k=0$, $h_{p,Lk} = h_{p,L}/(k,h_{p,L})$. Assume on the contrary that $h_{p,Lk}\mid j$, then
  \begin{equation}
    h_{p,L}\mid (k,h_{p,L})j \mid kj,
  \end{equation}
  a contradiction to $h_{p,L}\nmid kj$. Whence $h_{p,Lk}\nmid j$, i.e., $j\in S(A_k,B_k)$.
  \item[{($p\notin\{2,5\}$ and $D=[\mathrm{ii}]$)}] We have $(A,B) = (\{h_{p,L}\},\{h_{p^2,L}\})$. Thus $h_{p,L}\mid kj$ and $h_{p^2,L}\nmid kj$, and so $h_{p^2,L}\nmid k$. Consequently, by Lemma \ref{lem:rhodiv}, $p^2\nmid \rho_{k,L}$, and so $x_k\leq 1$.
  
  In case $x_k =0$, $D_k = [\mathrm{ii}]$ too, and so $(A_k,B_k) = (\{h_{p,Lk}\},\{h_{p^2,Lk}\})$. We have $h_{p,Lk} = h_{p,L}/(k,h_{p,L})$ and $h_{p^2,Lk} = h_{p^2,L}/(k,h_{p^2,L})$.
  Since $h_{p,L}\mid kj$, $h_{p,Lk}\mid j$.
  Assume on the contrary that $h_{p^2,Lk}\mid j$, then
  \begin{equation}
    h_{p^2,L}\mid (k,h_{p^2,L})j \mid kj,
  \end{equation}
  a contradiction to $h_{p^2,L}\nmid kj$. Whence $h_{p^2,Lk}\nmid j$, and so $j\in S(A_k,B_k)$.
  
  In case $x_k =1$, $D_k = [\mathrm{i}]$, and so $(A_k,B_k) = (\varnothing,\{h_{p,Lk}\})$. We have $h_{p,Lk} = h_{p^2,L}/(k,h_{p^2,L})$. 
  Assume on the contrary that $h_{p,Lk}\mid j$, then
  \begin{equation}
    h_{p^2,L}\mid (k,h_{p^2,L})j \mid kj,
  \end{equation}
  a contradiction to $h_{p^2,L}\nmid kj$. Whence $h_{p,Lk}\nmid j$, i.e., $j\in S(A_k,B_k)$.
  \item[{($p\notin\{2,5\}$ and $D=[\mathrm{iii}]$)}] We have $(A,B) = (\varnothing,\{h_{p^2,L}\})$. Thus $h_{p^2,L}\nmid kj$, and so $h_{p^2,L}\nmid k$. Consequently, by Lemma \ref{lem:rhodiv}, $p^2\nmid \rho_{k,L}$, and so $x_k\leq 1$.
  
  In case $x_k =0$, $D_k = [\mathrm{iii}]$ too, and so $(A_k,B_k) = (\varnothing,\{h_{p^2,Lk}\})$. We have $h_{p^2,Lk} = h_{p^2,L}/(k,h_{p^2,L})$. 
  Assume on the contrary that $h_{p^2,Lk}\mid j$, then
  \begin{equation}
    h_{p^2,L}\mid (k,h_{p^2,L})j \mid kj,
  \end{equation}
  a contradiction to $h_{p^2,L}\nmid kj$. Whence $h_{p^2,Lk}\nmid j$, i.e., $j\in S(A_k,B_k)$.
  
  In case $x_k =1$, $D_k = [\mathrm{i}]$, and so $(A_k,B_k) = (\varnothing,\{h_{p,Lk}\})$. We have $h_{p,Lk} = h_{p^2,L}/(k,h_{p^2,L})$.  Assume on the contrary that $h_{p,Lk}\mid j$, then
  \begin{equation}
    h_{p^2,L}\mid (k,h_{p^2,L})j \mid kj,
  \end{equation}
  a contradiction to $h_{p^2,L}\nmid kj$. Whence $h_{p,Lk}\nmid j$, i.e., $j\in S(A_k,B_k)$.
  \item[{($p\notin\{2,5\}$ and $D=[\mathrm{iv}]$)}] We have $(A,B) = (\{h_{p,L}\},\varnothing)$. Thus $h_{p,L}\mid kj$. 
  
  In case $x_k = 0$, $D_k = [\mathrm{iv}]$ too, and so $(A_k,B_k) = (\{h_{p,Lk}\},\varnothing)$. We have $h_{p,Lk} = h_{p,L}/(k,h_{p,L})$.
  Since $h_{p,L}\mid kj$, it is evident that $h_{p,Lk}\mid j$, i.e., $j\in S(A_k,B_k)$.
  
  In case $x_k \geq1 $, $D_k = [\mathrm{vi}]$, and so $(A_k,B_k) = (\varnothing,\varnothing)$. Whence $j\in S(A_k,B_k)$ holds trivially.
  \item[{($p\notin\{2,5\}$ and $D=[\mathrm{v}]$)}] We have $(A,B) = (\{h_{p^2,L}\},\varnothing)$. Thus $h_{p^2,L}\mid kj$.
  
  In case $x_k = 0$, $D_k = [\mathrm{v}]$ too, and so $(A_k,B_k) = (\{h_{p^2,Lk}\},\varnothing)$. We have $h_{p^2,Lk} = h_{p^2,L}/(k,h_{p^2,L})$. 
  Since $h_{p^2,L}\mid kj$, it is evident that $h_{p^2,Lk}\mid j$, i.e., $j\in S(A_k,B_k)$.
  
  In case $x_k = 1$, $D_k = [\mathrm{iv}]$, and so $(A_k,B_k) = (\{h_{p,Lk}\},\varnothing)$. We have $h_{p,Lk} = h_{p^2,L}/(k,h_{p^2,L})$. Since $h_{p^2,L}\mid kj$, it is evident that $h_{p,Lk}\mid j$, i.e., $j\in S(A_k,B_k)$.
  
  In case $x_k\geq2$, $D_k=[\mathrm{vi}]$, and so $(A_k,B_k) = (\varnothing,\varnothing)$. Whence $j\in S(A_k,B_k)$ holds trivially.
  \item[{($p\in\{2,5\}$ and $D\in\{[\mathrm{i}],[\mathrm{iii}]\}$)}] We have $(A,B)=(\varnothing,\varnothing)$. Since $p\in\{2,5\}$, $x_k=\ord_p(\rho_{k,L})=0$, and so $D_k=D$ still. Thus $(A_k,B_k) = (\varnothing,\varnothing)$. Whence $j\in S(A_k,B_k)$ holds trivially.
  \item[{($D=[\mathrm{vi}]$)}] We have $(A,B)=(\varnothing,\varnothing)$. We see that irregardless of $x_k$, $D_k=[\mathrm{vi}]$ too, and so $(A_k,B_k) = (\varnothing,\varnothing)$. Whence $j\in S(A_k,B_k)$ holds trivially.
  \item[{($p\in\{2,5\}$ and $D\in\{[\mathrm{ii}],[\mathrm{iv}],[\mathrm{v}]\}$, or $D=[\mathrm{vii}]$)}] We have $(A,B) = (\varnothing, \{1\})$. Thus $1\nmid kj$. But this is impossible, so actually this case cannot happen.
\end{description}

\bibliography{refs}{}

\begin{thebibliography}{9}

\bibitem{oeis} N. J. A. Sloane et al., The On-Line Encyclopedia of Integer Sequences, \url{https://oeis.org}, last accessed December 2021.

\bibitem {tsai} D.~Tsai, A recurring pattern in natural numbers of a certain property, \emph{Integers} {\bf 21} (2021), \#A32.

\bibitem {tsai18}  D.~Tsai, Natural numbers satisfying an unusual property, \emph{S\=ugaku Seminar} {\bf 57} (2018), 35--36 (written in Japanese).

\bibitem {tsai21}  D.~Tsai, The fundamental period of a periodic phenomenon pertaining to $v$-palindromes, preprint, 2021. Available at \url{http://arxiv.org/abs/2103.00989}.

\end{thebibliography}
\bibliographystyle{plain}

\end{document}